\documentclass[11pt]{amsart} 

\usepackage{amssymb,cite,hyperref}
\usepackage[margin=1in]{geometry}
\usepackage{enumitem,graphicx}
\usepackage{tensor}
\usepackage[all]{xy}
\SelectTips{cm}{}

\usepackage{colonequals}

\usepackage{todonotes}

\newcommand{\subgrp}[1]{\langle #1 \rangle}
\newcommand{\set}[1]{\left\{ #1 \right\}}
\newcommand{\abs}[1]{\left| #1 \right|}

\newcommand{\wt}[1]{\widetilde{ #1}}

\newcommand{\ol}[1]{\overline{#1}}
\newcommand{\wh}[1]{\widehat{ #1}}

\newcommand{\sslash}{/\!\!/}

\newcommand{\ev}{\textup{ev}}
\newcommand{\odd}{\textup{odd}}

\DeclareMathOperator{\ad}{ad}

\DeclareMathOperator{\chr}{char}

\DeclareMathOperator{\Ext}{Ext}
\DeclareMathOperator{\gldim}{gldim}
\DeclareMathOperator{\gr}{gr}
\DeclareMathOperator{\opH}{H}
\newcommand{\Hbul}{\opH^\bullet}
\DeclareMathOperator{\Hom}{Hom}
\DeclareMathOperator{\id}{id}
\DeclareMathOperator{\im}{im}

\DeclareMathOperator{\injdim}{injdim}

\DeclareMathOperator{\Max}{Max}

\DeclareMathOperator{\Nil}{Nil}
\DeclareMathOperator{\projdim}{projdim}

\DeclareMathOperator{\res}{res}

\DeclareMathOperator{\Span}{span}

\newcommand{\C}{\mathbb{C}}

\newcommand{\M}{\mathbb{M}}
\newcommand{\N}{\mathbb{N}}
\renewcommand{\P}{\mathbb{P}}
\newcommand{\Z}{\mathbb{Z}}

\newcommand{\Pone}{\P_1}
\newcommand{\Mones}{\M_{1;s}}

\newcommand{\g}{\mathfrak{g}}
\newcommand{\fm}{\mathfrak{m}}

\newcommand{\fp}{\mathfrak{p}}

\newcommand{\wtg}{\wt{\g}}
\newcommand{\wtM}{\wt{M}}
\newcommand{\wtN}{\wt{N}}

\newcommand{\Chi}{\mathcal{X}}

\newcommand{\calO}{\mathcal{O}}

\newcommand{\zero}{\ol{0}}
\newcommand{\one}{\ol{1}}

\newcommand{\Pbar}{\ol{P}}
\newcommand{\Ybar}{\ol{Y}}

\newcommand{\gone}{\g_{\one}}
\newcommand{\gzero}{\g_{\zero}}
\newcommand{\Vone}{V_{\one}}
\newcommand{\Vzero}{V_{\zero}}

\newcommand{\bfHom}{\mathbf{Hom}}

\newcommand{\GLmn}{GL_{m|n}}
\newcommand{\GLmnone}{GL_{m|n(1)}}

\newcommand{\gl}{\mathfrak{gl}}
\newcommand{\glmn}{\gl(m|n)}
\newcommand{\glone}{\gl(m|n)_{\one}}
\newcommand{\glzero}{\gl(m|n)_{\zero}}

\newcommand{\glmm}{\gl(m|m)}
\newcommand{\kmm}{k^{m|m}}

\numberwithin{equation}{subsection}

\newtheorem{theorem}{Theorem}[subsection]
\newtheorem*{theorem*}{Theorem}
\newtheorem{proposition}[theorem]{Proposition}
\newtheorem{corollary}[theorem]{Corollary}
\newtheorem{lemma}[theorem]{Lemma}

\theoremstyle{definition}
\newtheorem{definition}[theorem]{Definition}
\newtheorem*{definition*}{Definition}

\newtheorem{remark}[theorem]{Remark}

\title[Support varieties for modular Lie superalgebras]{Support varieties and modules of finite projective dimension for modular Lie superalgebras\\[1ex] \footnotesize\mdseries
  with an appendix on homological dimensions over Noether Algebras \\ by Luchezar L.\ Avramov and Srikanth B.\ Iyengar}

\author{Christopher M.\ Drupieski}
\address{Department of Mathematical Sciences,
		DePaul University,
		Chicago, IL 60614, USA}
\email{c.drupieski@depaul.edu}

\author{Jonathan R. Kujawa}
\address{Department of Mathematics \\
		University of Oklahoma \\
		Norman, OK 73019, USA}
\email{kujawa@math.ou.edu}

\thanks{The first author was supported in part by a Simons Collaboration Grant for Mathematicians, and by a Faculty Summer Research Grant from the DePaul University College of Science and Health in Summer 2019. The second author was supported in part by a Simons Collaboration Grant for Mathematicians.}

\subjclass[2010]{Primary 20G10. Secondary 17B56.}

\allowdisplaybreaks

\begin{document}

\begin{abstract}
We investigate cohomological support varieties for finite-dimensional Lie superalgebras defined over fields of odd characteristic. Verifying a conjecture from our previous work, we show the support variety of a finite-dimensional supermodule can be realized as an explicit subset of the odd nullcone of the underlying Lie superalgebra. We also show the support variety of a finite-dimensional supermodule is zero if and only if the supermodule is of finite projective dimension. As a consequence, we obtain a positive characteristic version of a theorem of B{\o}gvad, showing that if a finite-dimensional Lie superalgebra over a field of odd characteristic is absolutely torsion free, then its enveloping algebra is of finite global dimension.
\end{abstract}

\maketitle


\section{Introduction}

\subsection{Overview}

For more than three decades, the theory of cohomological support varieties has played an important role in non-semisimple representation theory. The foundations of the subject go back to the work of Quillen on the spectrum of the cohomology ring of a finite group \cite{Quillen:1971}, with the first fully-formed examples of the theory appearing a decade later through the work of Carlson and others in the context of finite groups \cite{Carlson:1983}, and through the work of Friedlander and Parshall in the context of finite-dimensional restricted Lie algebras \cite{Friedlander:1986b}. In these prototypical situations, the cohomology ring $\Hbul(G,k)$ of a group or restricted Lie algebra defines, via its maximal ideal spectrum, an affine algebraic variety $\abs{G}$. Then for each finite-dimensional $G$-module $M$, the support of the $\Hbul(G,k)$-module $\Ext_G^\bullet(M,M)$ defines a closed subvariety $\abs{G}_M$ of $\abs{G}$, called the \emph{support variety} of $M$. In these prototypical cases, one can describe $\abs{G}_M$ in terms of local representation-theoretic data (`rank varieties'), and it is thus possible to play the geometric structure of $\abs{G}_M$ and the representation-theoretic structure of $M$ off of each other, revealing new insights into both.

Betting on these successes it is perhaps natural to ask whether the technology of support varieties can be fruitfully applied to arbitrary Lie algebras. But here the theory stumbles out of the gate: the cohomology ring of a finite-dimensional Lie algebra is a finite-dimensional graded algebra, and hence its spectrum is trivial. In contrast, finite-dimensional Lie \emph{super}algebras do admit interesting support variety theories. For example, Boe, Nakano, and the second author introduced support varieties for classical Lie super\-algebras over the field of complex numbers, and showed these varieties encode representation-theoretic information, including the atypicality and complexity of modules, and the thick tensor ideals of the module category \cite{Boe:2009, Boe:2010, Boe:2012, Boe:2017}. Their varieties were based not on the `ordinary' cohomology ring $\Hbul(\g,\C)$, but on the relative cohomology ring $\Hbul(\g,\gzero;\C)$. In independent work, Duflo and Serganova \cite{Duflo:2005} defined non-cohomological associated varieties for Lie superalgebras in characteristic zero and showed that they also encode representation-theoretic data. In both of these approaches, the characteristic zero hypothesis was essential.

This paper is a continuation of our work investigating the representation theory and cohomology of Lie super\-algebras in positive characteristic. Let $k$ be an algebraically closed field of odd characteristic, and let $\g = \gzero \oplus \gone$ be a finite-dimensional Lie superalgebra over $k$. In our previous work \cite{Drupieski:2019a}, we showed that the spectrum of $\Hbul(\g,k)$ is homeomorphic to the \emph{odd nullcone} of $\g$,
	\[
	\Chi_\g(k) = \set{x \in \gone : [x,x] = 0};
	\]
see Theorem \ref{thm:homeomorphism}.\footnote{There is an extra Frobenius twist involved, which we ignore for the purposes of the introduction.} We also defined the rank variety of a supermodule: Given a $\g$-supermodule $M$ and a nonzero element $x \in \Chi_\g(k)$, let $M|_{\subgrp{x}}$ denote the restriction of $M$ to the one-dimensional purely odd Lie sub-superalgebra of $\g$ generated by $x$. Since $[x,x] = 0$, the enveloping superalgebra $U(\subgrp{x})$ is then isomorphic to an exterior algebra on one generator. We say that $M|_{\subgrp{x}}$ is \emph{free} if $M$ is free as a $U(\subgrp{x}) \cong \Lambda(x)$-supermodule. The \emph{rank variety} for $M$ is then defined to be the set
	\[
	\Chi_{\g}'(M) = \set{x \in \Chi_\g(k) : M|_{\subgrp{x}} \text{ is not free}} \cup \set{0}.
	\]
We observed that the rank variety is always a subset of the cohomological support variety $\Chi_\g(M)$, and based on partial results and examples we conjectured that the two always coincide. Having an explicit, non-cohomological description of the support variety is invaluable for both calculational and theoretical purposes. For example, taking the equality $\Chi_\g'(M) = \Chi_\g(M)$ for granted, one can deduce the \emph{tensor product property}, namely, for any finite-dimensional $\g$-supermodules $M$ and $N$,
	\[
	\Chi_{\g}(M \otimes N) = \Chi_{\g}(M) \cap \Chi_{\g}(N).
	\]

\subsection{Main results}

The first main result of this paper is the verification of the conjecture from the previous paragraph. Given a finite-dimensional $\g$-supermodule $M$, we show in Corollary~\ref{cor:rank=support} that the cohomological support variety of $M$ is equal to the rank variety of $M$:
	\[
	\Chi_{\g}'(M) = \Chi_{\g}(M).
	\]
Similar to the situation for restricted Lie algebras \cite{Friedlander:1986b}, one can reduce the proof of this equality to the case when $M$ is the natural supermodule for the general linear Lie superalgebra $\g = \glmn$. From there, however, the proof requires substantially different arguments than from the classical case. One important tool is a spectral sequence that relates the cohomology of $\g$ to the cohomology of an associated graded Lie superalgebra $\wtg$ arising from the Clifford filtration on $\g$; see Sections \ref{subsec:clifford} and \ref{subsec:cliffordspecseq} for details. While the original Lie superalgebra $\g$ may not be a restricted Lie superalgebra, the associated graded Lie superalgebra $\wtg$ is close to being abelian, and we can use $\wtg$ to define a related $p$-nilpotent restricted Lie superalgebra for which $\wtM$ (the associated graded supermodule arising from a choice of standard filtration on $M$) is a restricted supermodule. Equivalently, $\wtM$ becomes a supermodule for a certain height-one infinitesimal unipotent supergroup scheme. Then applying our previous work on support varieties for infinitesimal unipotent supergroup schemes \cite{Drupieski:2018}, we can place an upper bound on the size of $\Chi_\g(M)$, and hence show that $\Chi_\g'(M) = \Chi_\g(M)$.

Our second main result is a characterization of when the support variety of a supermodule is trivial. In Theorem~\ref{theorem:0varietyimpliesfinprojdim} we prove for a finite-dimensional $\g$-supermodule $M$ that
	\begin{equation} \label{eq:zerosupport}
	\Chi_{\g}(M) =  \set{0} \quad \text{if and only if} \quad \projdim_{U(\g)}(M) < \infty.
	\end{equation}
The proof of \eqref{eq:zerosupport} relies on some results of Avramov and Iyengar, which the authors have kindly included here in Appendix \ref{SS:homologicaldimension}. Using the equality $\Chi_\g'(M) = \Chi_\g(M)$, \eqref{eq:zerosupport} can be rephrased as:
	\[
	\projdim_{U(\g)}(M) < \infty \quad \text{if and only if} \quad \set{ x \in \gone : [x,x] = 0 \text{ and } M|_{\subgrp{x}} \text{ is not free}} = \emptyset.
	\]
Thus, the finitude of projective dimension can be detected in terms of local representation-theoretic data. In contrast to the situations for finite groups or restricted Lie algebras (or more generally, for finite-dimensional self-injective algebras), for Lie superalgebras it is possible for a supermodule to have a finite but nonzero projective dimension. For example, if $N$ is a $\gzero$-module, then the induced supermodule $U(\g) \otimes_{U(\gzero)} N$ has finite projective dimension but is rarely projective.

Using \eqref{eq:zerosupport}, we deduce in Corollary \ref{C:finiteglobaldimension} that a finite-dimensional Lie super\-algebra over a field odd characteristic has finite global dimension if and only if its odd nullcone is zero. This result was proved previously for finite-dimensional solvable (positively graded) Lie superalgebras by B{\o}gvad \cite{Bo-gvad:1984} and for arbitrary finite-dimensional Lie superalgebras in characteristic zero by Musson \cite{Musson:2012}. Musson also makes use of the Clifford filtration and Lie superalgebra cohomology, but our approach is different from his.

\subsection{Further questions}

The rank variety $\Chi_\g'(M)$, which is equal to the cohomological support variety $\Chi_\g(M)$ by Corollary~\ref{cor:rank=support}, is equivalent in definition to the associated variety defined by Duflo and Serganova for Lie superalgebras in characteristic zero. Moreover, the fact that $\Chi_\g(M)$ detects when a supermodule is of finite projective dimension parallels the projectivity detection result of \cite[Theorem 3.11]{Duflo:2005}. These facts are quite striking given that, to our knowledge, there is no known cohomological realization for Duflo and Serganova's associated varieties in characteristic zero. The similarities between the two theories warrant further investigation. It would also be interesting to use the rank variety description to compute support varieties for interesting families of supermodules such as Kac supermodules, simple supermodules, and so on.

Another standard question in support variety theory is that of \emph{realization}, namely, given a closed conical subvariety $W$ of the cohomological spectrum $\Chi_\g(k)$, is there a finite-dimensional $\g$-super\-module $M$ such that $W = \Chi_\g(M)$? Since the projectives for the enveloping algebra $U(\g)$ are infinite-dimensional, one cannot simply imitate the standard argument involving Carlson's $L_\zeta$-modules (e.g., as presented in \cite[\S2]{Feldvoss:2010}). But if one is willing to consider finitely-generated $U(\g)$-supermodules, then, for example, we expect the approach of Avramov and Iyengar \cite{Avramov:2007} would adapt to this setting.


\subsection{Acknowledgements} \label{subsection:acknowledgments}

The authors thank Luchezar Avramov and Srikanth Iyengar for their generosity in sharing their preprint \cite{Avramov:2014} and for contributing Appendix \ref{SS:homologicaldimension}.

\subsection{Conventions} \label{subsection:conventions}

We generally follow the conventions of our previous work \cite{Drupieski:2019a,Drupieski:2019b}, to which we refer the reader for any unexplained terminology or notation. Throughout, $k$ will denote a field of characteristic $p \geq 3$ (assumed to be algebraically closed beginning in Section \ref{section:supportvarieties}), all vector spaces will be $k$-vector spaces, all algebras will be $k$-algebras, and all unadorned tensor products will be tensor products over $k$. Given a $k$-vector space $V$, let $V^* = \Hom_k(V,k)$ be its $k$-linear dual, and let $V^{(1)} = V \otimes_{\varphi} k$ be its Frobenius twist, i.e., the $k$-vector space obtained via base change along the Frobenius morphism $\varphi: \lambda \mapsto \lambda^p$. Given $v \in V$, set $v^{(1)} = v \otimes_{\varphi} 1 \in V^{(1)}$. More generally, if $X$ is an affine $k$-scheme (resp.\ algebraic variety) with coordinate algebra $k[X]$, we write $X^{(1)}$ for the scheme (resp.\ variety) with coordinate algebra $k[X^{(1)}] = k[X]^{(1)}$. If $Y = X^{(1)}$, then we will find it convenient to write $X = Y^{(-1)}$.

Set $\Z_2 = \Z/2\Z = \set{\zero,\one}$. Following the literature, we use the prefix `super' to indicate that an object is $\Z_2$-graded. We frequently leave the prefix ``super'' implicit when the $\Z_{2}$-grading is clear from context.  We denote the decomposition of a vector superspace into its $\Z_2$-homogeneous components by $V = \Vzero \oplus \Vone$, calling $\Vzero$ and $\Vone$ the even and odd subspaces of $V$, respectively, and writing $\ol{v} \in \Z_2$ to denote the superdegree of a homogeneous element $v \in V$. Whenever we state a formula in which homogeneous degrees are specified, we mean that the formula is true as written for homogeneous elements and that it extends linearly to non-homogeneous elements. We use the symbol $\cong$ to denote even (i.e., degree-preserving) isomorphisms of superspaces, and use $\simeq$ for odd (i.e., degree-reversing) isomorphisms. In the context of algebraic varieties, we use $\cong$ to denote an isomorphism of varieties, and use $\simeq$ to denote a homeomorphism of topological spaces.

We use the adjective \emph{graded} to indicate that an object admits an additional $\Z$-grading that is compatible with its underlying structure. Thus a \emph{graded superspace} is a $(\Z \times \Z_2)$-graded vector space, a \emph{graded superalgebra} is a $(\Z \times \Z_2)$-graded algebra, etc. Given a graded superspace $V$ and a homogeneous element $v \in V$ of bidegree $(s,t) \in \Z \times \Z_2$, we write $\deg(v) = s$ and $\ol{v} = t$ for the $\Z$-degree and the $\Z_2$-degree of $v$, respectively. Then if $A$ is a graded superalgebra, we say that $A$ is \emph{graded commutative} provided that for all homogeneous elements $a,b \in A$, one has
	\begin{equation} \label{eq:gradedcommutative}
	ab = (-1)^{\deg(a) \cdot \deg(b) + \ol{a} \cdot \ol{b}}ba.
	\end{equation}
More generally, we use the sign convention indicated in \eqref{eq:gradedcommutative} whenever homogeneous bigraded symbols pass each other. So for example, if $A$ and $B$ are graded superalgebras, then $A \otimes B$ is also a graded superalgebra, with product defined by $(a \otimes b) (c \otimes d) = (-1)^{\deg(b) \cdot \deg(c) + \ol{b}\cdot \ol{c}} ac \otimes bd$.

Let $\N = \set{0,1,2,3,\ldots}$ be the set of non-negative integers.

\section{Preliminaries} \label{section:preliminaries}

We begin in Section \ref{subsec:LSAcohomology} by recalling some basic facts concerning the calculation of Lie super\-algebra cohomology. In Sections \ref{subsec:clifford} and \ref{subsec:cliffordspecseq} we recall the definition of the Clifford filtration on a Lie superalgebra, and establish some preliminary results concerning a spectral sequence that arises from the Clifford filtration. Then in Section \ref{subsec:homdimenveloping} we collect some results that will allow us to apply Theorem \ref{thm:equivprojdimfinite} to the enveloping superalgebra of a finite-dimensional Lie superalgebra over $k$.

Throughout this section, let $\g = \gzero \oplus \gone$ be a finite-dimensional Lie superalgebra over $k$.

\subsection{Lie superalgebra cohomology} \label{subsec:LSAcohomology}

Let $Y(\g) = U(\g) \# \Ybar(\g)$ be the Koszul resolution of $\g$ as described in \cite[\S3.1]{Drupieski:2013c}. As a graded super\-algebra, it is a smash product of the enveloping algebra $U(\g)$, considered as a graded super\-algebra concentrated in $\Z$-degree $0$, and the graded super\-algebra $\Ybar(\g) = \Lambda(\gzero) \otimes \Gamma(\gone)$. Here $\Gamma(\gone)$ denotes the ordinary divided power algebra on the $k$-vector space $\gone$, and the component of $\Z$-degree $n$ in $\Ybar(\g)$ is given by $\Ybar_n(\g) = \bigoplus_{i+j=n} \Lambda^i(\gzero) \otimes \Gamma^j(\gone)$. The coproducts on $U(\g)$, $\Lambda(\gzero)$, and $\Gamma(\gone)$ induce on $Y(\g)$ the structure of a graded superbialgebra. Then the differential on $Y(\g)$ makes $Y(\g)$ into a differential graded superbialgebra and into a $U(\g)$-free resolution of the trivial module $k$. Denoting a monomial in $Y(\g) = U(\g) \# \left( \Lambda(\gzero) \otimes \Gamma(\gone) \right)$ by $u \subgrp{x_{i_1} \ldots x_{i_b}} \gamma_{a_1}(y_1) \cdots \gamma_{a_t}(y_t)$ as in \cite{Drupieski:2013c}, the differential $d: Y(\g) \rightarrow Y(\g)$ is defined on algebra generators by the formulas
	\begin{align*}
	d(u) &= 0, \\
	d(\subgrp{x}) &= x, \quad \text{and} \\
	d(\gamma_r(y)) &= y \gamma_{r-1}(y) - \tfrac{1}{2} \subgrp{[y,y]} \gamma_{r-2}(y).
	\end{align*}
An explicit formula for the map $d: Y_n(\g) \rightarrow Y_{n-1}(\g)$ is given in \cite[Remark 3.1.4]{Drupieski:2013c}.

For any pair of $\g$-modules $M$ and $N$, one gets that $Y(\g) \otimes M$ is a $U(\g)$-projective resolution of $M$, and the cohomology group $\Ext_\g^\bullet(M,N)$ can be computed as the cohomology of
	\[
	C^\bullet(\g,M,N) \colonequals \Hom_{U(\g)}(Y_\bullet(\g) \otimes M,N) \cong \Hom_k(\Ybar_\bullet(\g) \otimes M,N).
	\]
In particular, the cohomology ring $\Hbul(\g,k) = \Ext_{\g}^\bullet(k,k)$ can be computed as the cohomology of the cochain complex $C^\bullet(\g,k) \colonequals \Hom_{U(\g)}(Y_\bullet(\g),k) \cong \Hom_k(\Ybar_\bullet(\g),k)$.

The coalgebra structure of $Y(\g)$ induces an algebra structure on $C^\bullet(\g,k)$, and $C^\bullet(\g,k)$ is then isomorphic as a graded superalgebra to $\Lambda_s(\g^*)$, the superexterior algebra on $\g^*$ \cite[Lemma 3.2.1]{Drupieski:2013c}. As a graded super\-algebra, $\Lambda_s(\g^*) = \Lambda(\gzero^*) \otimes S(\gone^*)$, with $\g^* = \gzero^* \oplus \gone^*$ concentrated in $\Z$-degree $1$. The differential $\partial$ on $C^\bullet(\g,k)$ makes $\Lambda_s(\g^*)$ into a differential graded superalgebra, and the map $\partial: \g^* = \Lambda_s^1(\g^*) \rightarrow \Lambda_s^2(\g^*) \cong [\Lambda_s^2(\g)]^*$ then identifies with the transpose of the Lie bracket. In particular, the product on $\Lambda_s(\g^*)$ descends to the cup product in the cohomology ring $\Hbul(\g,k)$. More generally, let $\Delta: Y(\g) \rightarrow Y(\g) \otimes Y(\g)$ be the coproduct on $Y(\g)$, which is a map of chain complexes, and let $\Delta_{i,j}: Y_{i+j}(\g) \rightarrow Y_i(\g) \otimes Y_j(\g)$ be the evident component. Then given $f \in C^i(\g,k)$ and $g \in C^j(\g,M,N)$, define $f \odot g \in C^{i+j}(\g,M,N)$ by
	\begin{equation} \label{eq:fodotg}
	f \odot g \colonequals (f \otimes g) \circ (\Delta_{i,j} \otimes \id_M) : Y_{i+j}(\g) \otimes M \rightarrow k \otimes N = N.
	\end{equation}
If $f$ and $g$ are cocycles representing classes $\alpha \in \Ext_{\g}^i(k,k)$ and $\beta \in \Ext_{\g}^j(M,N)$, respectively, then $f \odot g$ is a cocycle representative for the cup product $\alpha \cup \beta \in \Ext_{\g}^{i+j}(M,N)$.

Write $\gone^*[p]$ for $\gone^*$ considered as a graded superspace concentrated in $\Z$-degree $p$. The map $z^{(1)} \mapsto z^p$ extends to an algebra map $S(\gone^*[p])^{(1)} \rightarrow S(\gone^*)$, and composing with the inclusion $S(\gone^*) \hookrightarrow \Lambda_s(\g^*)$, this produces an injective homomorphism of graded superalgebras
	\begin{equation} \label{eq:whvarphi}
	\wh{\varphi} = \wh{\varphi}_\g: S(\gone^*[p])^{(1)} \rightarrow \Lambda_s(\g^*).
	\end{equation}
Since $\Lambda_s(\g^*)$ is a graded-commutative superalgebra, and since the differential $\partial$ on $\Lambda_s(\g^*)$ acts by derivations, it follows that the image of $\wh{\varphi}$ consists of cocycles, and hence that $\wh{\varphi}$ induces a graded superalgebra homomorphism
	\begin{equation} \label{eq:varphi}
	\varphi = \varphi_\g: S(\gone^*[p])^{(1)} \rightarrow \Hbul(\g,k).
	\end{equation}
By \cite[Theorem 3.2.4]{Drupieski:2013c}, $\Hbul(\g,k)$ is finite over the image of $\varphi$, and if $M$ and $N$ are finite-dimensional $\g$-modules, then $\Ext_{\g}^\bullet(M,N) \cong \Ext_{\g}^\bullet(k,\Hom_k(M,N))$ is a finite $\Hbul(\g,k)$-module.

\subsection{Clifford filtration} \label{subsec:clifford}

The Clifford filtration on $\g$ is the increasing Lie superalgebra filtration $0 = F^0 \g \subseteq F^1 \g \subseteq F^2 \g = \g$ defined by $F^1 \g = \gone$. The associated graded Lie superalgebra,
	\[
	\wtg \colonequals \gr(\g) = \left( F^2 \g / F^1 \g \right) \oplus \left( F^1\g / F^0 \g \right) = \wtg_2 \oplus \wtg_1,
	\]
identifies with $\g$ as a vector superspace, with $\gone$ now in $\Z$-degree $1$ and $\gzero$ in $\Z$-degree $2$. Under this identification, the Lie bracket on $\wtg_1$ identifies with the original Lie bracket on $\gone$, and $\wtg_2$ is central in $\wtg$. The Clifford filtration on $\g$ induces an increasing nonnegative filtration $F^\bullet U(\g)$ on $U(\g)$ such that $F^0 U(\g) = k$; we call this the Clifford filtration on $U(\g)$. Then $\gr U(\g) = U(\wtg)$.

Let $N$ be a $\g$-module, and let $S \subseteq N$ be a $U(\g)$-module generating set for $N$. Then $N$ admits an increasing nonnegative filtration $F^\bullet N$, which we call the standard filtration associated to $S$, defined by
	\begin{equation} \label{eq:modulefiltration}
	F^i N = ( F^i U(\g)). S.
	\end{equation}
Thus for all $i,j \in \N$, one has $(F^i U(\g)).(F^j N) \subseteq F^{i+j} N$. We denote by $\wtN = \gr(N)$ the associated graded module of $N$. Evidently, if $S$ is a finite generating set for $N$, then $\wtN$ is a finitely-generated $\wtg$-module, generated by any basis for $\wtN_0$. 


\subsection{A spectral sequence} \label{subsec:cliffordspecseq}

In this section let $M$ be a finite-dimensional $\g$-module, and let $N$ be a finitely-generated $\g$-module. We consider $\g$ as filtered by the Clifford filtration, and we assume that $M$ and $N$ are equipped with standard filtrations $F^\bullet M$ and $F^\bullet N$ as in \eqref{eq:modulefiltration}, associated to some fixed choices of finite generating sets. Let $\wtg$ be the associated graded Lie superalgebra of $\g$, and let $\wtM$ and $\wtN$ be the associated graded $\wtg$-modules. In the special case of the trivial $\g$-module $k$, we have $F^0k = k$, so that $\wt{k}$ is concentrated in $\Z$-degree $0$.

The Clifford filtration on $\g$ induces an increasing filtration $F^\bullet Y(\g)$ on the Koszul complex $Y(\g)$, which is preserved by the differential in the sense that $d(F^i Y_n(\g)) \subseteq F^i Y_{n-1}(\g)$, and is compatible with the coproduct $\Delta: Y(\g) \rightarrow Y(\g) \otimes Y(\g)$ in the sense that $\Delta(F^\ell Y(\g)) \subseteq \sum_{i+j=\ell} F^i Y(\g) \otimes F^j Y(\g)$. The projective resolution $P_\bullet \colonequals Y_\bullet(\g) \otimes M$ of $M$ inherits a filtration defined by $F^\ell P_n = \sum_{i+j = \ell} F^i Y_n(\g) \otimes F^j M$, and the associated graded complex $\gr(P_\bullet)$ then identifies with the $U(\wtg)$-projective resolution $Y(\wtg) \otimes \wtM$ of $\wtM$. Evidently $\Pbar_\bullet \colonequals \Ybar_\bullet(\g) \otimes M$ inherits a filtration from $P_\bullet$ by restriction. Now define a decreasing filtration on $C^\bullet(\g,M,N)$ by
	\[
	F^i C^n(\g,M,N) = \set{ f \in C^n(\g,M,N) : f(F^j P_n) \subseteq F^{j-i}N \text{ for all $j \in \N$}}.
	\]
Given $x \in \g$, $u \in Y_n(\g)$, and $m \in M$, one has
	\[
	(x.u) \otimes m = x.(u \otimes m) - (-1)^{\ol{x} \cdot \ol{u}} u \otimes (x.m),
	\]
and from this it follows that the filtration on $C^n(\g,M,N)$ is also given by
	\[
	F^i C^n(\g,M,N) = \set{ f \in C^n(\g,M,N) : f(F^j \Pbar_n) \subseteq F^{j-i}N \text{ for all $j \in \N$}}.
	\]
If $f \in F^i C^m(\g,k)$ and $g \in F^j C^n(\g,M,N)$, then $f \odot g \in F^{i+j} C^{m+n}(\g,M,N)$, so the cup product of cochains makes $C^\bullet(\g,k)$ into a filtered differential graded algebra, and makes $C^\bullet(\g,M,N)$ into a filtered differential graded module over $C^\bullet(\g,k)$. For fixed $i$, one gets
	\begin{equation} \label{eq:quotientcomplex}
	F^i C^\bullet(\g,M,N) / F^{i+1} C^\bullet(\g,M,N) \cong C^\bullet(\wtg,\wtM,\wtN)_{-i} = \Hom_{U(\wtg)}(Y_\bullet(\wtg) \otimes \wtM, \wtN)_{-i},
	\end{equation}
where $\Hom_{U(\wtg)}(Y_n(\wtg) \otimes \wtM, \wtN)_{-i}$ denotes the set of $U(\wtg)$-module homomorphisms such that
	\[
	f([Y_n(\wtg) \otimes \wtM]_j) \subseteq \wtN_{j-i} \quad \text{ for all $j \in \N$}.
	\]


Let $n \in \N$ and let $f \in C^n(\g,M,N)$. Since $\Pbar_n$ and hence also $f(\Pbar_n)$ are finite-dimensional, it follows that there exist integers $s(n)$ and $t(f,n)$ such that $F^{s(n)} \Pbar_n = \Pbar_n$ and $f(\Pbar_n) \subseteq F^{t(f,n)} N$. Then $f \in F^{-t(f,n)} C^n(\g,M,N)$, since for all $j \in \N$ one gets
	\[
	f(F^j \Pbar_n) \subseteq f(\Pbar_n) \subseteq F^{t(f,n)}N \subseteq F^{j+t(f,n)}N.
	\]
Then $C^n(\g,M,N) = \bigcup_{i \in \Z} F^i C^n(\g,M,N)$, so the filtration on $C^n(\g,M,N)$ is exhaustive. (If there exists $t \in \N$ such that $F^t N = N$, e.g., if $N$ is finite-dimensional, then the filtration is not just exhaustive but is also bounded below.) On the other hand, if $f \in F^{s(n)+1}C^n(\g,M,N)$, then
	\[
	f(\Pbar_n) = f(F^{s(n)} \Pbar_n) \subseteq F^{s(n)-(s(n)+1)} N = F^{-1}N = 0,
	\]
so $f = 0$ and hence $F^{s(n)+1} C^n(\g,M,N) = 0$. Thus for each fixed $n$, the filtration on $C^n(\g,M,N)$ is bounded above. Then by \cite[Theorem 3.2]{McCleary:2001} and \eqref{eq:quotientcomplex}, there exists a spectral sequence
	\[
	E_1^{i,j}(M,N) = \Ext_{\wtg}^{i+j}(\wtM,\wtN)_{-i} \Rightarrow \Ext_{\g}^{i+j}(M,N), \tag*{$E(M,N):$}
	\]
where the subscript $-i$ denotes the component of internal $\Z$-degree $-i$. Since the filtrations on $C^\bullet(\g,k)$ and $C^\bullet(\g,M,N)$ are compatible with the cup products of cochains, it follows that $E(k,k)$ is a spectral sequence of algebras, and $E(M,N)$ is a spectral sequence of  over $E(k,k)$.

Observe that $\Ybar_n(\wtg) = \bigoplus_{s+t=n} \Lambda^s(\wtg_2) \otimes \Gamma^t(\wtg_1)$, and $\Lambda^s(\wtg_2) \otimes \Gamma^t(\wtg_1)$ is concentrated in $\Z$-degree $2s+t$. Then $\Ybar_n(\wtg)_i \neq 0$ only if $n \leq i \leq 2n$. Consequently, $E_1^{i,j}(k,k) \neq 0$ only if $i+j \geq 0$ and $-\frac{1}{2}i \leq j \leq 0$, so $E(k,k)$ is concentrated in the fourth quadrant. In particular, $E_\infty^{i,j}(k,k) = 0$ for $j > 0$. Since the filtration $F^\bullet \opH^n(\g,k)$ on $\opH^n(\g,k)$ arising from the $E_\infty$-page of $E(k,k)$ is bounded, this implies that $F^i \opH^n(\g,k) = \opH^n(\g,k)$ for $i \leq n$, and $E_\infty^{n,0}(k,k) = \opH^n(\g,k) / F^{n+1} \opH^n(\g,k)$. Then one gets an algebra homomorphism
	\begin{equation}
	\pi: \Hbul(\g,k) \rightarrow \Hbul(\wtg,k)_{-\bullet},
	\end{equation}
which is defined in cohomological degree $n$ by the composition
	\[
	\opH^n(\g,k) \twoheadrightarrow \opH^n(\g,k)/F^{n+1}\opH^n(\g,k) = E_\infty^{n,0} \hookrightarrow E_1^{n,0} = \opH^n(\wtg,k)_{-n},
	\]
where the unlabeled arrows are the canonical maps. Now $\pi$ fits in the commutative diagram
	\begin{equation} \label{eq:varphipi}
	\vcenter{\xymatrix{
	S(\gone^*[p])^{(1)} \ar@{->}[r]^{\varphi_\g} \ar@{->}[d]^{\cong} & \Hbul(\g,k) \ar@{->}[d]^{\pi} \\
	S(\wtg_1^*[p])^{(1)} \ar@{->}[r]^{\varphi_{\wtg}} & \Hbul(\wtg,k),
	}}
	\end{equation}
in which the left-hand vertical arrow is induced by the canonical identification $\gone \cong \wtg_1$. To check the commutativity of \eqref{eq:varphipi}, it suffices to check commutativity on the subspace $\gone^*[p]^{(1)}$, and this can be verified at the level of cochains by verifying the commutativity of the diagram
	\begin{equation} \label{eq:whvarphidiagram}
	\vcenter{\xymatrix{
	\gone^*[p]^{(1)} \ar@{->}[r]^{\wh{\varphi}_\g} \ar@{->}[d]^{\cong} & F^p C^p(\g,k)\ar@{->>}[r] & F^p C^p(\g,k) / F^{p+1} C^p(\g,k) \ar@{->}[d]^{\cong} \\
	\wtg_1^*[p]^{(1)} \ar@{->}[rr]^{\wh{\varphi}_{\wtg}} &  & C^p(\wtg,k)_{-p},
	}}
	\end{equation}
in which the right-hand vertical arrow is \eqref{eq:quotientcomplex}. The commutativity of \eqref{eq:varphipi} implies that the image of $\varphi_{\wtg}: S(\wtg_1^*[p])^{(1)} \rightarrow \Hbul(\wtg,k)$ consists of permanent cycles in the row $j = 0$ of $E_1(k,k)$.

In the next lemma we show that $\Ext_{\wtg}^\bullet(\wtM,\wtN)$ is finite under the cup product action of $\Hbul(\wtg,k)$. When $N$ is finite-dimensional this follows already from \cite[Theorem 3.2.4]{Drupieski:2013c} and the isomorphism $\Ext_{\wtg}^\bullet(\wtM,\wtN) \cong \Ext_{\wtg}^\bullet(k,\Hom_k(\wtM,\wtN))$, so the content of the lemma is in the case when $N$ is finitely-generated but not finite-dimensional.

\begin{lemma} \label{lemma:E1pagefinite}
Retain the notation and assumptions from the first paragraph of this section. Then $\Ext_{\wtg}^\bullet(\wtM,\wtN)$ is a finite module under the cup product action of $\Hbul(\wtg,k)$.
\end{lemma}

\begin{proof}
As $\Hbul(\wtg,k)$-modules one has $\Ext_{\wtg}^\bullet(\wtM,\wtN) \cong \Ext_{\wtg}^\bullet(k,\Hom_k(\wtM,\wtN))$, and since $\wtM$ is finite-dimensional one has $\Hom_k(\wtM,\wtN) \cong \wtN \otimes \wtM^*$ as $U(\wtg)$-modules. By assumption, $\wtN_0$ generates $\wtN$ and is finite-dimensional. Then it follows that $\wtN \otimes \wtM^*$ is generated as a $U(\wtg)$-module by the finite-dimensional subspace $\wtN_0 \otimes \wtM^*$, and hence $\wtN \otimes \wtM^*$ is finitely-generated. Thus for the remainder of the proof we may assume that $M = k$.

First consider the case $\wtg = \wtg_2$. Then $\wtg$ is abelian, and $U(\wtg)$ is isomorphic to a polynomial ring over $k$ in $\dim_k(\wtg)$ variables. In particular, $U(\wtg)$ is a commutative noetherian ring. The Koszul resolution $Y_\bullet(\wtg)$ is a resolution of $k$ by finitely-generated free $U(\wtg)$-modules; say $Y_i(\wtg) \cong U(\wtg)^{\oplus n_i}$. Then $\opH^i(\wtg,\wtN)$ is a $U(\wtg)$-module subquotient of $\Hom_{U(\wtg)}(Y_i(\wtg),\wtN) \cong \wtN^{\oplus n_i}$. Since $\wtN$ is a finitely-generated $U(\wtg)$-module, this implies that $\opH^i(\wtg,\wtN)$ is a noetherian $U(\wtg)$-module. On the other hand, if $\wtN \rightarrow Q^\bullet$ is a $U(\wtg)$-injective resolution of $\wtN$, then $\Hbul(\wtg,\wtN)$ can be computed as the cohomology of the cochain complex $\Hom_{U(\wtg)}(k,Q^\bullet) = (Q^\bullet)^{U(\wtg)}$. Computing this way, we see that the $U(\wtg)$-action on each cohomology group is trivial. Thus each $\opH^i(\wtg,\wtN)$ is a finitely-generated trivial $U(\wtg)$-module, hence a finite-dimensional $k$-vector space. Finally, since a polynomial ring over $k$ in $n$ variables has global dimension $n$, we get that $\opH^i(\wtg,\wtN) = 0$ for $i > n$. Thus $\Hbul(\wtg,\wtN)$ is finite-dimensional, so in particular is finite under the cup product action of $\Hbul(\wtg,k)$.

Now for the case of general $\wtg$, observe that since $U(\wtg)$ is finite over the subalgebra $U(\wtg_2)$, the module $\wtN$ is finitely-generated over $U(\wtg_2)$. The algebra $U(\wtg_2)$ is a central Hopf subalgebra of $U(\wtg)$, and the Hopf superalgebra quotient $U(\wtg) \sslash U(\wtg_2)$ is isomorphic to the exterior algebra $\Lambda(\wtg_1)$, which in turn identifies with the enveloping algebra of a purely odd abelian Lie superalgebra. The extension of Hopf superalgebras $U(\wtg_2) \hookrightarrow U(\wtg) \twoheadrightarrow \Lambda(\wtg_1)$ gives rise to the LHS spectral sequence
	\begin{equation} \label{eq:wtgLHSspecseq}
	E_2^{i,j}(N) = \opH^i(\Lambda(\wtg_1), \opH^j(U(\wtg_2),\wtN)) \Rightarrow \opH^{i+j}(U(\wtg),\wtN).
	\end{equation}
By the purely even case of the previous paragraph, $\Hbul(U(\wtg_2),\wtN)$ is finite-dimensional. Then by the purely odd case of \cite[Theorem 3.2.4]{Drupieski:2013c}, the $E_2$-page of \eqref{eq:wtgLHSspecseq} is a finite module over the algebra $E_2^{\bullet,0}(k) = \Hbul(\Lambda(\wtg_1),k)$. This implies by \cite[Lemma 1.6]{Friedlander:1997} that $\Hbul(U(\wtg),\wtN)$ is finite over the image of the inflation map $\Hbul(\Lambda(\wtg_1),k) \rightarrow \Hbul(U(\wtg),k)$, so in particular is finite over $\Hbul(U(\wtg),k)$.
\end{proof}

\begin{proposition} \label{prop:EMNfinitenessresults}
Let $\g$ be a finite-dimensional Lie superalgebra over $k$. Let $M$ be a finite-dimen\-sional $\g$-module, and let $N$ be a finitely-generated $\g$-module, considered as filtered via the standard filtrations associated to some fixed choices of finite generating sets. Then:
	\begin{enumerate}
	\item \label{item:stops} The spectral sequence $E(M,N)$ is concentrated in only finitely many rows, and hence stops after finitely many pages.
	\item \label{item:filtrationlayerzero} Writing $F^\bullet \Ext_{\g}^n(M,N)$ for the filtration on $\Ext_{\g}^n(M,N)$ coming from the $E_\infty$-page of $E(M,N)$, there exists an integer $L(M,N) \in \N$ such that
		\[
		F^{n+\ell} \Ext_\g^n(M,N) = 0 \quad \text{for} \quad \ell > L(M,N).
		\]
	\item \label{item:Extfinite} $\Ext_{\g}^\bullet(M,N)$ is finitely-generated under the cup product action of $\Hbul(\g,k)$.
	\end{enumerate}
\end{proposition}

\begin{proof}
By Lemma \ref{lemma:E1pagefinite}, we get that $E_1(M,N) = \Ext_{\wtg}^\bullet(\wtM,\wtN)$ is finite under the cup product action of $E_1(k,k) = \Hbul(\wtg,k)$, which in turn is finite over the image of $\varphi_{\wtg}: S(\wtg_1^*[p]) \rightarrow \Hbul(\wtg,k)$. We may choose a homogeneous finite generating set for $E_1(M,N)$, which will then be concentrated in only finitely many rows. Since $\im(\varphi_{\wtg})$ is contained in the row $j=0$ of $E_1(k,k)$, this implies that $E_1(M,N)$ is concentrated in the same finitely many rows as the generating set. Then there exists an integer $L(M,N) \in \N$ such that $E_1^{i,j}(M,N) = 0$ for $\abs{j} > L(M,N)$, and hence
	\begin{equation} \label{eq:Einfty=0}
	F^i \Ext_\g^{i+j}(M,N) / F^{i+1} \Ext_\g^{i+j}(M,N) = E_\infty^{i,j}(M,N) = 0 \quad \text{for $\abs{j} > L(M,N)$.}
	\end{equation}
The filtration on $\Ext_{\g}^n(M,N)$ coming from the $E_\infty$-page of $E(M,N)$ is bounded above (because the original filtration on $C^n(\g,M,N)$ is bounded above), so $F^i \Ext_\g^n(M,N) = 0$ for $i \gg 0$. Then taking $i = n+\ell$ and $j = -\ell$, \eqref{eq:Einfty=0} implies that $F^{n+\ell} \Ext_\g^n(M,N) = 0$ for $\ell > L(M,N)$. Finally, by the commutativity of the diagram \eqref{eq:varphipi}, the image $\varphi_{\wtg}: S(\wtg_1^*[p]) \rightarrow \Hbul(\wtg,k)$ consists of permanent cycles in $E_1(k,k)$. Since $E_r(M,N)$ is a subquotient of $E_{r-1}(M,N)$, it follows by induction on $r$ that each page of $E(M,N)$ is a noetherian $S(\wtg_1^*[p])^{(1)}$-module. By \eqref{item:stops}, we may choose $r \gg 0$ such that $E_r(k,k) = E_\infty(k,k)$ and $E_r(M,N) = E_\infty(M,N)$. Then $E_\infty(M,N) = \gr(\Ext_{\g}^\bullet(M,N))$ is noetherian over $\gr(\im(\varphi_{\g})) \subseteq \gr(\Hbul(\g,k))$. Since the filtration on each $\Ext_{\g}^n(M,N)$ is bounded above, this implies by \cite[Lemma 7.4.5]{Evens:1991} that $\Ext_\g^\bullet(M,N)$ is noetherian over $\im(\varphi_{\g})$, hence finitely-generated under the cup product action of $\Hbul(\g,k)$.
\end{proof}

In the last lemma of this section we specialize to the case $M = N$.

\begin{lemma} \label{lemma:powerannihilates}
Let $f \in S(\gone^*[p])^{(1)}$ be a homogeneous polynomial, considered also as a polynomial in $S(\wtg_1^*[p])^{(1)}$ via the canonical identification $S(\gone^*[p])^{(1)} \cong S(\wtg_1^*[p])^{(1)}$. Suppose $\varphi_{\wtg}(f) \cup \id_{\wtM} = 0$ in $\Ext_{\wtg}^\bullet(\wtM,\wtM)$. Then $\varphi_\g(f^\ell) \cup \id_M = 0$ in $\Ext_\g^\bullet(M,M)$ for $\ell > L(M,M)$.
\end{lemma}

\begin{proof}
The homomorphism $\pi \circ \varphi_\g: S(\gone^*[p])^{(1)} \rightarrow \Hbul(\wtg,k)$ makes $E(M,M)$ into a spectral sequence of modules over the algebra $S(\gone^*[p])^{(1)}$. Given $\alpha \in \Hbul(\wtg,k)$ and $\beta \in \Ext_{\wtg}^\bullet(\wtM,\wtM)$, one has
	\[
	\alpha \cup \beta = (\alpha \cup \id_{\wtM}) \circ \beta,
	\]
where $\circ$ denotes the Yoneda composition of extensions, so the hypothesis $\varphi_{\wtg}(f) \cup \id_{\wtM} = 0$ implies that the cup product action of $f$ on $E_1(M,M) = \Ext_{\wtg}^\bullet(\wtM,\wtM)$ is identically zero. Then the action of $f$ on all subsequent pages of $E(M,M)$ is also zero. Since $E(M,M)$ is concentrated in only finitely many rows, one has $E_r(M,M) = E_\infty(M,M)$ for all $r \gg 0$, so the cup product action of $f$ on $E_\infty(M,M)$ is identically zero. Say $f$ has polynomial degree $t$, so that $\varphi_\g(f) \in \opH^{pt}(\g,k) = F^{pt} \opH^{pt}(\g,k)$. Then given $\alpha \in F^i \Ext_\g^n(M,M)$, one must have
	\begin{equation} \label{eq:factionfiltrationdegree}
	\varphi_\g(f) \cup \alpha \in F^{i+pt+1} \Ext_\g^{n+pt}(M,M) \subseteq F^{i+pt} \Ext_\g^{n+pt}(M,M).
	\end{equation}
Since $\id_M \in F^0 \Ext_\g^0(M,M)$, one can inductively apply \eqref{eq:factionfiltrationdegree} to get
	\[
	\varphi_\g(f^\ell) \cup \id_M = \varphi_\g(f)^\ell \cup \id_M \in F^{pt\ell + \ell} \Ext_\g^{pt\ell}(M,M)
	\]
for all $\ell \in \N$. Then by Proposition \ref{prop:EMNfinitenessresults}\eqref{item:filtrationlayerzero}, $\varphi_\g(f^\ell) \cup \id_M = 0$ if $\ell > L(M,M)$.
\end{proof}

\subsection{Homological dimensions for enveloping superalgebras} \label{subsec:homdimenveloping}

In this section we collect some results that will enable us to apply Theorem \ref{thm:equivprojdimfinite} to the enveloping algebra of a finite-dimensional Lie superalgebra.

\begin{lemma} \label{lemma:centralsubalgebra}
Let $\g$ be a finite-dimensional Lie superalgebra over a field $k$ of characteristic $p \geq 3$. There exists a purely even central subalgebra $\calO \subseteq U(\g)$ such that $\calO$ is isomorphic to a polynomial ring in $\dim_k(\gzero)$ variables and $U(\g)$ is a free $\calO$-module of finite rank.
\end{lemma}

\begin{proof}
It follows from the PBW theorem for Lie superalgebras (see Theorem 3.2.2 and Remark 3.2.3 of \cite{Bahturin:1992}) that the enveloping algebra $U(\g)$ is free of finite rank over the subalgebra $U(\gzero)$, so it suffices to exhibit a polynomial subalgebra $\calO \subseteq U(\gzero)$ such that $\calO$ is central in $U(\g)$ and $U(\gzero)$ is free of finite rank over $\calO$. For this one can follow the proof of \cite[Theorem 5.1.2]{Strade:1988}, considering $\ad(e)$ for $e \in \gzero$ as an endomorphism of the finite-dimensional $k$-vector space $\g$.
\end{proof}

Given a $k$-superalgebra $A$, one can form the smash product algebra $A \# k\Z_2$. As a vector space, $A \# k\Z_2$ is equal to $A \otimes_k k\Z_2$, the tensor product of $A$ and the group ring $k\Z_2$. Multiplication in $A \# k\Z_2$ is induced by the products in $A$ and $k\Z_2$ and by the relation $(1 \otimes \one)(a \otimes \zero) = (-1)^{\ol{a}} a \otimes \one$. If $A$ is a Hopf superalgebra, then $A \# k\Z_2$ becomes an ordinary Hopf algebra; cf.\ \cite[\S10.6]{Montgomery:1993}.

\begin{lemma} \label{lem:smashproductONoether}
Let $\g$ be a finite-dimensional Lie superalgebra over a field $k$ of characteristic $p \geq 3$, and let $\calO \subseteq U(\g)$ be a central subalgebra as in \ref{lemma:centralsubalgebra}. Then $U(\g) \# k\Z_2$ is a Noether $\calO$-algebra in the sense of Definition \ref{def:noetheralgebra}. In particular, $U(\g) \# k\Z_2$ is a noetherian PI Hopf algebra.
\end{lemma}

\begin{proof}
Since the algebra $\calO$ of Lemma \ref{lemma:centralsubalgebra} is purely even, its image in $U(\g) \# k\Z_2$ remains central, and $U(\g) \# k\Z_2$ is then a free $\calO$-module of finite rank, so $U(\g) \# k\Z_2$ is a Noether $\calO$-algebra. Then $U(\g) \# k\Z_2$ is a noetherian PI Hopf algebra by \cite[Corollary 13.1.13(iii)]{McConnell:2001}.
\end{proof}

Again let $A$ be a $k$-superalgebra. Each $A$-supermodule $M$ lifts to an $A \# k\Z_2$-module by having the non-identity element $\one \in \Z_2 \subset k\Z_2$ act via the sign automorphism $m \mapsto (-1)^{\ol{m}}m$. Conversely, since $\chr(k) \neq 2$, each $A \# k\Z_2$-module $M$ decomposes under the action of $\Z_2$ into a trivial isotypic component $M_{\zero}$ and a nontrivial isotypic component $M_{\one}$, and the decomposition $M = M_{\zero} \oplus M_{\one}$ then makes $M$ into an $A$-supermodule. Via this equivalence, it follows that an $A$-supermodule is injective (resp.\ projective) if and only if it is injective (resp.\ projective) as an $A \# k\Z_2$-module, and given $A$-supermodules $M$ and $N$, one gets $\Ext_{A \# k\Z_2}^\bullet(M,N) = \Ext_A^\bullet(M,N)_{\zero}$; for further discussion, see \cite[\S2.3]{Drupieski:2019a}.

Recall that an algebra $A$ is Gorenstein if it has finite injective dimension as a left or right module over itself.

\begin{lemma}
Let $\g$ be a finite-dimensional Lie superalgebra over a field $k$ of characteristic $p \geq 3$. Then $U(\g)$ is a noetherian Gorenstein algebra, and for each finite $U(\g)$-supermodule $M$ one has
	\[
	\projdim_{U(\g)}(M) < \infty \quad \text{if and only if} \quad \injdim_{U(\g)}(M) < \infty.
	\]
\end{lemma}

\begin{proof}
We apply the results of Wu and Zhang \cite{Wu:2003} and the discussion preceding the lemma to the noetherian PI Hopf algebra $U(\g) \# k\Z_2$. Each irreducible $U(\g) \# k\Z_2$-module is finite-dimensional over $k$ by \cite[Lemma 2.4]{Bahturin:1992}, so $U(\g) \# k\Z_2$ is then Gorenstein by \cite[Theorem 3.5]{Wu:2003}. The second assertion of the lemma is now a consequence of the first by \cite[Lemma 4.2]{Wu:2003}.
\end{proof}

\section{Support varieties} \label{section:supportvarieties}

In Section \ref{subsec:preliminariessupport} we recall the definition of cohomological support varieties and some basic results about support varieties for restricted and non-restricted Lie superalgebras. Then in Sections \ref{subec:supportvargraded} and \ref{subsec:supportvarungraded} we prove a rank variety description for the support varieties of finite-dimensional supermodules, first in the context of certain graded Lie superalgebras, and then for arbitrary finite-dimensional Lie superalgebras over $k$. In Section \ref{subsec:applications} we present our main applications, including a characterization of when a finite-dimensional $\g$-supermodule has finite projective dimension.

From now on we assume that $k$ is an algebraically closed field of characteristic $p \geq 3$.

\subsection{Preliminaries on support varieties} \label{subsec:preliminariessupport}

Given a graded superalgebra $R = \bigoplus_{n \in \Z} R^n$ that is graded-commutative in the sense of \eqref{eq:gradedcommutative}, set
	\[
	\ol{R} = \bigoplus_{n \in \Z} \left( (R^{2n})_{\zero} \oplus (R^{2n+1})_{\one} \right) = R_{\zero}^{\ev} \oplus R_{\one}^{\odd}.
	\]
Then $\ol{R}$ is a graded ring that is commutative in the non-graded sense, and \cite[Corollary 2.2.5]{Drupieski:2019a} implies that the inclusion $\ol{R} \hookrightarrow R$ induces an isomorphism $\ol{R}/\Nil(\ol{R}) \cong R/\Nil(R)$. We then define the maximal ideal spectrum of $R$,
	\[
	\Max(R) = \Max(\ol{R}),
	\]
to be the set of maximal ideals of the ring $R/\Nil(R)$, considered as a topological space via the Zariski topology. If $R$ is a finitely-generated $k$-algebra (as will be the case in all of our particular situations of interest), then $\Max(R)$ is an affine algebraic variety.

If $A$ is a Hopf superalgebra over $k$, then the cohomology ring $\Hbul(A,k)$ is a graded-commutative superalgebra by \cite[Corollary 2.3.6]{Drupieski:2019a}. In this case we set $H(A,k) = \ol{\Hbul(A,k)}$, and we write
	\[
	\abs{A} = \Max\left( \Hbul(A,k) \right) = \Max \left( H(A,k) \right)
	\]
for the cohomological spectrum of $A$. Given a left $A$-module $M$, let $I_A(M)$ be the annihilator ideal for the left cup product action of $\Hbul(A,k)$ on $\Ext_A^\bullet(M,M)$. Equivalently, $I_A(M)$ is the kernel of the superalgebra homomorphism
	\[
	\Phi_M: \Hbul(A,k) = \Ext_A^\bullet(k,k) \rightarrow \Ext_A^\bullet(M,M)
	\]
induced by the tensor product functor $- \otimes M$. Then the support variety $\abs{A}_M$ is defined by
	\[
	\abs{A}_M = \Max\left( \Hbul(A,k) / I_A(M) \right).
	\]

Now let $\g$ be a finite-dimensional Lie superalgebra over $k$, and let $\varphi : S(\gone^*[p])^{(1)} \rightarrow \Hbul(\g,k)$ be the graded superalgebra homomorphism of \eqref{eq:varphi}. Given a $\g$-module $M$, set $I_\g(M) = I_{U(\g)}(M)$, and let $J_\g(M) = \varphi^{-1}(I_\g(M))$. One has $I_\g(k) = \set{0}$, so $\ker(\varphi) = J_\g(k) \subseteq J_\g(M)$. We now define $\Chi_\g(M)$ to be the closed subvariety of $\Chi_\g(k) \colonequals \Max( S(\gone^*[p])^{(1)} / J_\g(k))$ defined by $J_\g(M)$,
	\begin{equation} \label{eq:ChigMdef}
	\Chi_\g(M) = \Max \left( S(\gone^*[p])^{(1)}/J_\g(M) \right).
	\end{equation}
Since $S(\gone^*)^{(1)} \cong S((\gone^*)^{(1)}) \cong S((\gone^{(1)})^*)$, $\Chi_\g(M)$ identifies with a subset of $\gone^{(1)}$. The next theorem is a composite of Proposition 4.2.2, Corollary 4.2.3, and Theorem 4.2.4 of \cite{Drupieski:2019a}.\footnote{In \cite[\S\S4--5]{Drupieski:2019a}, we erroneously omitted Frobenius twists from the descriptions of certain varieties; we correct those omissions in this section.}

\begin{theorem} \label{thm:homeomorphism}
For each $\g$-module $M$, the induced homomorphism
	\[
	\varphi_M: S(\gone^*[p])^{(1)}/J_\g(M) \rightarrow \Hbul(\g,k)/I_\g(M)
	\]
is injective modulo nilpotents and is surjective onto $p$-th powers, and thus induces a homeomorphism
	\[
	\varphi_M^*: \abs{U(\g)}_M \simeq \Chi_\g(M).
	\]
The variety $\Chi_\g(k)$ is isomorphic to the Frobenius twist of the odd nullcone of $\g$, i.e.,
	\begin{equation} \label{eq:oddnullconeiso}
	\Chi_\g(k)^{(-1)} \cong \set{ x \in \gone : [x,x]=0}.
	\end{equation}
\end{theorem}

From now one we may make the identification in \eqref{eq:oddnullconeiso} without further comment. Then given $x \in \Chi_\g(k)^{(-1)}$ and a $\g$-module $M$, let $M|_{\subgrp{x}}$ denote the restriction of $M$ to the $k$-sub\-algebra $\Lambda(x)$ of $U(\g)$ generated by $x$. If $x \neq 0$, then $\Lambda(x)$ is an exterior algebra generated by $x$, while if $x=0$, then $\Lambda(x) = k$. We say that $M|_{\subgrp{x}}$ is \emph{free} if $M$ is free as a $\Lambda(x)$-module, and we let $\Chi_\g'(M)$ be the subvariety of $\Chi_\g(k)$ defined by
	\begin{equation} \label{eq:rankvarietydef}
	\Chi_{\g}'(M)^{(-1)} = \set{x \in \Chi_\g(k)^{(-1)} : M|_{\subgrp{x}} \text{ is not free}} \cup \set{0}.
	\end{equation}
Given $m \in \N$, let $\kmm$ denote the natural representation of $\glmm$. The following proposition is then a composite of Propositions 4.2.2 and 4.3.1 of \cite{Drupieski:2019a}.
	
\begin{proposition} \label{prop:rankinclusion}
Let $M$ be a finite-dimen\-sional $\g$-module. Then
	\begin{equation} \label{eq:rankinclusion}
	\Chi_{\g}'(M) \subseteq \Chi_\g(M).
	\end{equation}
If the equality $\Chi_{\glmm}'(\kmm) = \Chi_{\glmm}(\kmm)$ holds for all $m \in \N$, then \eqref{eq:rankinclusion} is an equality.
\end{proposition}

Now suppose that $\g$ is a finite-dimensional restricted Lie superalgebra, with restricted enveloping algebra $V(\g) = U(\g)/\subgrp{x^p-x^{[p]}: x \in \gzero}$. By \cite[Remark 4.4.4]{Drupieski:2013c}, there exists a height-one infinitesimal supergroup scheme $G$ such that $kG \colonequals k[G]^* = V(\g)$. The module categories for $V(\g)$ and $G$ are equivalent, so henceforth we make the identifications $\Hbul(V(\g),k) = \Hbul(G,k)$ and $\abs{V(\g)}_M = \abs{G}_M$; cf.\ \cite[\S4.3 and Remark 5.1.1]{Drupieski:2013c}.

Fix a closed embedding $\iota: G \hookrightarrow \GLmnone$ of $G$ into the first Frobenius kernel of the general linear supergroup $\GLmn$, for some $m,n \in \N$. By abuse of notation, we also write $\iota: \g \rightarrow \glmn$ for the induced embedding of restricted Lie superalgebras. Then by \cite[Proposition 5.2.5]{Drupieski:2019a}, there exists a commutative diagram of graded superalgebra homomorphisms
	\begin{equation} \label{eq:phidiagram}
	\vcenter{\xymatrix{
	S(\glzero^*[2] \oplus \glone^*[p])^{(1)} \ar@{->}[r]^-{\phi} \ar@{->>}[d]^{\iota^*} & \Hbul(\GLmnone,k) \ar@{->}[d]^{\iota^*} \\
	S(\gzero^*[2] \oplus \gone^*[p])^{(1)} \ar@{->}[r]^-{\phi_G} & \Hbul(G,k).
	}}
	\end{equation}
Here $\phi$ is the map denoted $\phi_{GL_{m|n(1)}}$ in \cite{Drupieski:2019a}, and $\phi_G$ is the map obtained in \cite[Proposition 5.2.5]{Drupieski:2019a} from the factorization of $\iota^* \circ \phi$ through the restriction map $\iota^*: S(\glmn^*) \rightarrow S(\g^*)$.

Next let $\Pone = k[u,v]/\subgrp{u^p+v^2}$ be the Hopf superalgebra over $k$ generated by the even primitive element $u$ and the odd primitive element $v$ such that $uv = vu$ and $u^p + v^2 = 0$. Given $s \geq 1$, let $k\Mones = \Pone / \subgrp{u^{p^s}} = k[u,v]/\subgrp{u^{p^s},u^p+v^2}$, let $V_{1;s}(G) = [\bfHom(\Mones,G)]_{\ev}$ be the affine $k$-scheme defined in \cite[\S3.3]{Drupieski:2019b}, and let $\psi_{1;s}: H(G,k) \rightarrow k[V_{1;s}(G)]$ be the homomorphism of graded $k$-algebras defined in \cite[\S6.2]{Drupieski:2019b}. By \cite[Lemma 6.2.1]{Drupieski:2019b}, \eqref{eq:phidiagram} can be extended to the commutative diagram
	\begin{equation} \label{eq:phipsidiagram}
	\vcenter{\xymatrix{
	S(\glmn^*)^{(1)} \ar@{->}[r]^-{\phi} \ar@{->>}[d]^{\iota^*} & H(\GLmnone,k) \ar@{->}[d]^{\iota^*} \ar@{->}[r]^-{\psi_{1;s}} & k[V_{1;s}(\GLmnone)] \ar@{->>}[d]^{\iota^*} \\
	S(\g^*)^{(1)} \ar@{->}[r]^-{\phi_G} & H(G,k) \ar@{->}[r]^-{\psi_{1;s}} & k[V_{1;s}(G)],
	}}
	\end{equation}
in which the right-hand vertical arrow is a surjection by \cite[Theorem 3.3.6(2)]{Drupieski:2019b}. The corresponding diagram of varieties then has the form
	\begin{equation} \label{eq:phipsidiagramvarieties}
	\vcenter{\xymatrix{
	\glmn^{(1)} \ar@{<-}[r]^-{\phi^*} \ar@{<-}[d]^{\iota^{(1)}} & \abs{\GLmnone}  \ar@{<-}[d]^{\iota_*} \ar@{<-}[r]^-{\psi_{1;s}^*} & V_{1;s}(\GLmnone) \ar@{<-}[d]^{\iota_*} \\
	\g^{(1)} \ar@{<-}[r]^-{\phi_G^*} & \abs{G} \ar@{<-}[r]^-{\psi_{1;s}^*} & V_{1;s}(G),
	}}
	\end{equation}
A point $\nu$ of the variety $V_{1;s}(G)$ is a homomorphism of $k$-supergroup schemes $\nu: \Mones \rightarrow G$, or equivalently a Hopf superalgebra homomorphism $\nu: k\Mones \rightarrow kG = V(\g)$. Such a $\nu$ is specified by the even primitive element $\nu(u) \in V(\g)$ and the odd primitive element $\nu(v) \in V(\g)$. Then $\nu(u) \in \gzero$ and $\nu(v) \in \gone$ by \cite[Theorem 3.2.11]{Bahturin:1992}\footnote{As stated, the cited theorem imposes the additional assumption that $p > 3$. However, if one assumes, as we do, that the identity $[y,[y,y]] = 0$ holds in $\g$ for all $y \in \gone$, then the same argument also goes through in the case $p=3$; cf.\ Section 1.1.10 and Remarks 3.2.3 and 3.2.8 of \cite{Bahturin:1992}.}, so one gets an identification of varieties
	\begin{equation} \label{eq:V1sidentification}
	V_{1;s}(G) \cong \set{ (\alpha,\beta) \in \gzero \oplus \gone : \alpha^{[p^s]} = 0 \text{ and } \alpha^{[p]} + \tfrac{1}{2}[\beta,\beta] = 0},
	\end{equation}
and similarly for $V_{1;s}(\GLmnone)$. Then by \cite[Theorem 6.2.3]{Drupieski:2019b}, the composite morphism
	\[
	\phi^* \circ \psi_{1;s}^*: V_{1;s}(\GLmnone) \rightarrow \abs{\GLmnone} \rightarrow \glmn^{(1)}
	\]
identifies with the composition of the inclusion $V_{1;s}(\GLmnone) \subseteq \glmn$ and the map $\glmn \rightarrow \glmn^{(1)}$ defined by $z \mapsto z^{(1)}$.\footnote{There is a natural identification $\glmn^{(1)} \rightarrow \glmn$, defined by sending $\alpha^{(1)} = \alpha \otimes_{\varphi} k$ to the matrix obtained by raising each individual matrix entry of $\alpha$ to the $p$-th power; see \cite[Remark 5.1.6]{Drupieski:2019b}. Making this identification, the morphism $\phi^* \circ \psi_{1;s}^*: V_{1;s}(\GLmnone) \rightarrow \glmn$ then matches the exact description given in \cite[Theorem 6.2.3]{Drupieski:2019b}.} This implies by commutativity of \eqref{eq:phipsidiagramvarieties} that the morphism
	\[
	\phi_G^* \circ \psi_{1;s}^*: V_{1;s}(G) \rightarrow \g^{(1)}
	\]
identifies with the composition of the inclusion $V_{1;s}(G) \subseteq \g$ and the map $\g \rightarrow \g^{(1)}$, $z \mapsto z^{(1)}$.

\begin{lemma} \label{lemma:phiGappliedtosupport}
Suppose that $G$ is unipotent, so that the augmentation ideal $I$ of $kG$ is nilpotent, and assume that $I^{p^s} = 0$. Let $M$ be a finite-dimensional $G$-module. Given a point $(\alpha,\beta) \in V_{1;s}(G)$ as in \eqref{eq:V1sidentification}, let $\sigma_{(\alpha,\beta)}: \Pone \twoheadrightarrow k\Mones \rightarrow kG = V(\g)$ be the unique Hopf superalgebra homomorphism such that $\sigma_{(\alpha,\beta)}(u) = \alpha$ and $\sigma_{(\alpha,\beta)}(v) = \beta$, and let $\sigma_{(\alpha,\beta)}^*M$ be the $\Pone$-module obtained from $M$ by pulling back along $\sigma_{(\alpha,\beta)}$. Then
	\[
	\phi_G^*(\abs{G}_M)^{(-1)} = \set{ (\alpha,\beta) \in V_{1;s}(G) : \projdim_{\Pone}(\sigma_{(\alpha,\beta)}^* M) = \infty}.
	\]
\end{lemma}

\begin{proof}
Let $V_1(G) = [\bfHom(\Pone,kG)]_{\ev}$ be the affine $k$-scheme defined in \cite[\S4.1]{Drupieski:2018}, whose $k$-points correspond to Hopf superalgebra homomorphisms $\nu: \Pone \rightarrow kG$. The assumptions on $G$ and $s$ imply by \cite[Lemma 4.1.6]{Drupieski:2018} that the quotient map $q: \Pone \twoheadrightarrow k\Mones$ induces an identification $k[V_1(G)] = k[V_{1;s}(G)]$. Via this identification, the homomorphism $\psi_1: H(G,k) \rightarrow k[V_1(G)]$ defined in \cite[\S4.2]{Drupieski:2018} identifies with the map $\psi_{1;s}: H(G,k) \rightarrow k[V_{1;s}(G)]$; see the proof of \cite[Lemma 4.2.7]{Drupieski:2018}. Then \eqref{eq:phipsidiagram} can be rewritten with $k[V_{1;s}(G)]$ replaced by $k[V_1(G)]$, and the corresponding diagram of varieties takes the form
	\begin{equation} \label{eq:V1Gdiagram}
	\vcenter{\xymatrix{
	\glmn^{(1)} \ar@{<-}[r]^-{\phi^*} \ar@{<-}[d]^{\iota^{(1)}} & \abs{\GLmnone}  \ar@{<-}[d]^{\iota_*} \ar@{<-}[r]^-{\psi_{1;s}^*} & V_{1;s}(\GLmnone) \ar@{<-}[d]^{\iota_*} \\
	\g^{(1)} \ar@{<-}[r]^-{\phi_G^*} & \abs{G} \ar@{<-}[r]^-{\psi_1^*} & V_1(G).
	}}
	\end{equation}
Now $\psi_1^*: V_1(G) \rightarrow \abs{G}$ is a homeomorphism of varieties by \cite[Theorem 5.1.3]{Drupieski:2018}, and by \cite[Lemma 5.2.2]{Drupieski:2018}, this homeomorphism maps the support set $V_1(G)_M$ defined in \cite[\S4.3]{Drupieski:2018} bijectively onto the support variety $\abs{G}_M$. Then by the definition of $V_1(G)_M$ and by the description of the composite morphism $\phi_G^* \circ \psi_{1;s}^* = \phi_G^* \circ \psi_1^*$ given prior to the lemma,
	\[
	\phi_G^*(\abs{G}_M)^{(-1)} = (\phi_G^* \circ \psi_1^*)(V_1(G)_M)^{(-1)} = \set{ (\alpha,\beta) \in V_{1;s}(G) : \projdim_{\Pone}(\sigma_{(\alpha,\beta)}^* M) = \infty}. \qedhere
	\]
\end{proof}

\subsection{Support varieties in the graded case} \label{subec:supportvargraded}

In this section we assume that $\g = \g_1 \oplus \g_2$ is a finite-dimensional $\Z$-graded Lie superalgebra such that $\g_2$ is central in $\g$, and we let $M = \bigoplus_{i \in \Z} M_i$ be a finite-dimensional $\Z$-graded $\g$-module. Since $M$ is finite-dimensional, it is concentrated in only finitely many integer degrees, say, $M_i \neq 0$ only if $\abs{i} < m$. One has $x.M_i \subseteq M_{i+2}$ for each $x \in \g_2$, so this implies that the action of the enveloping algebra $U(\g)$ on $M$ factors through the quotient $U(\g)/\subgrp{x^{p^m}: x \in \g_2}$. Let $\wtg$ be the restricted Lie superalgebra generated by the image of $\g$ in $U(\g)/\subgrp{x^{p^m}: x \in \g_2}$. Then $\wtg$ is a $\Z$-graded Lie superalgebra, and one has
	\[
	\wtg = \wtg_1 \oplus \wtg_2 \oplus \wtg_{2p} \oplus \cdots \oplus \wtg_{2p^{m-1}},
	\]
where $\wtg_1 = \g_1$, and $\wtg_{2p^i}$ is the subspace of $U(\g)/\subgrp{x^{p^m}: x \in \g_2}$ spanned by the monomials of the form $x^{p^i}$ for $x \in \g_2$. In particular, $\g = \wtg_1 \oplus \wtg_2$ as nonrestricted Lie superalgebras, and $\bigoplus_{i=1}^{m-1} \wtg_{2p^i}$ is a central ideal in $\wtg$. By the definition of $\wtg$, the $\g$-module structure on $M$ lifts to $\wtg$.

\begin{lemma} \label{lemma:supportvarietygradedcase}
Retain the assumptions and notation of the first paragraph of this section. Then
	\[
	\Chi_{\wtg}'(M) = \Chi_{\wtg}(M).
	\]
\end{lemma}

\begin{proof}
Let $V(\wtg)$ be the restricted enveloping algebra of $\wtg$. By \cite[Remark 4.4.3]{Drupieski:2013c}, there exists a height-one infinitesimal supergroup $G$ such that $kG \colonequals k[G]^* = V(\wtg)$. Then we can consider $M$ as a rational $G$-module, and cohomology for $V(\wtg)$ identifies with cohomology for $G$. Since the $p$-map on $\wtg$ is nilpotent, it follows that $G$ is unipotent. As in Section \ref{subsec:preliminariessupport}, fix a choice of closed embedding $\iota: G \hookrightarrow \GLmnone$, for some $m,n \in \N$.

Let $\pi = \pi_{\wtg}: U(\wtg) \twoheadrightarrow U(\wtg)/\subgrp{x^p-x^{[p]}: x \in \wtg_{\zero}} = V(\wtg)$ be the quotient homomorphism from the universal enveloping algebra of $\wtg$ to the restricted enveloping algebra of $\wtg$. Then one gets the following cube of algebra homomorphisms:
	\begin{equation} \label{eq:cube}
	\vcenter{\xymatrix{
	S(\glmn^*)^{(1)} \ar@{->}[rr]^-{\phi} \ar@{->>}[rd]^-{\res} \ar@{->>}[dd]^(.7){\iota^*} && \Hbul(\GLmnone,k) \ar@{->>}[dd]|-(.49){{\phantom{ll}}}^(.7){\iota^*} \ar@{->}[rd]^{\pi_{\glmn}^*} & \\
	& S(\glone^*)^{(1)} \ar@{->}[rr]^(.35){\varphi_{\glmn}} \ar@{->>}[dd]^(.7){\iota^*} && \Hbul(\glmn,k) \ar@{->}[dd]^(.7){\iota^*} \\
		S(\wtg^*)^{(1)} \ar@{->>}[rd]^-{\res} \ar@{->}[rr]|-(.49){{\phantom{X}}}^(0.6){\phi_G} && \Hbul(G,k) \ar@{->}[rd]^{\pi_{\wtg}^*} & \\
	& S(\wtg_{\one}^*)^{(1)} \ar@{->}[rr]^-{\varphi_{\wtg}} && \Hbul(\wtg,k).
	}}
	\end{equation}
The back face of the cube commutes by \eqref{eq:phidiagram}, and the front face commutes by the naturality of \eqref{eq:varphi}. The left face evidently commutes, and the right face commutes because the supergroup homomorphism $\iota: G \hookrightarrow \GLmnone$ differentiates to a map of restricted Lie superalgebras. The top face commutes by the left-hand square of \cite[(5.4.4)]{Drupieski:2019a}, and then a diagram chase using the surjectivity of the indicated maps implies that the bottom face commutes as well. Now commutativity of the bottom face implies that the following diagram of varieties also commutes:
	\begin{equation} \label{eq:pivarietydiagram}
	\vcenter{\xymatrix{
	\wtg^{(1)} \ar@{<-}[r]^-{\phi_G^*} \ar@{<-^)}[d]^{\res^*} & \abs{G} \ar@{<-}[d]^{\pi_*} \ar@{<-^)}[r] & \abs{G}_M \ar@{<-}[d]^{\pi_*} \\
	\wtg_{\one}^{(1)} \ar@{<-}[r]^-{\varphi_{\wtg}^*}_{\simeq} & \abs{U(\wtg)} \ar@{<-^)}[r] & \abs{U(\wtg)}_M,
	}}
	\end{equation}
The bottom left arrow of \eqref{eq:pivarietydiagram} is a homeomorphism by Theorem \ref{thm:homeomorphism}. Then
	\[
	\res^*(\Chi_\g(M)) = ({\res^*} \circ {\varphi_{\wtg}^*})(\abs{U(\wtg)}_M) = ({\phi_G^*} \circ {\pi_*})(\abs{U(\wtg)}_M) \subseteq \phi_G^*(\abs{G}_M).
	\]
This implies by Lemma \ref{lemma:phiGappliedtosupport} and \cite[Lemma 4.3.3]{Drupieski:2018} that
	\begin{align*}
	\Chi_{\wtg}(M)^{(-1)} = \res^*(\Chi_{\wtg}(M))^{(-1)} &\subseteq \set{ (0,\beta) \in V_{1;s}(G) : \projdim_{\Pone}(\sigma_{(0,\beta)}^* M) = \infty} \\
	&= \set{ \beta \in \wtg_{\one} : [\beta,\beta]=0 \text{ and } M|_{\subgrp{\beta}} \text{ is not free}}.
	\end{align*}
So $\Chi_{\wtg}(M) \subseteq \Chi_{\wtg}'(M)$. Since also $\Chi_{\wtg}'(M) \subseteq \Chi_{\wtg}(M)$ by \eqref{eq:rankinclusion}, then $\Chi_{\wtg}'(M) = \Chi_{\wtg}(M)$.
\end{proof}

\begin{lemma} \label{lemma:rankequalssupportgradedcase}
Retain the assumptions and notation of the first paragraph of this section. In particular, make the identification $\g_1 = \wtg_1$. Then
	\[
	\Chi_\g'(M) = \Chi_\g(M) = \Chi_{\wtg}(M).
	\]
\end{lemma}

\begin{proof}
Naturality of \eqref{eq:varphi} implies commutativity of the diagram
	\begin{equation}
	\vcenter{\xymatrix{
	S(\wtg_1^*[p])^{(1)} \ar@{->}[r]^-{\varphi_{\wtg}} \ar@{=}[d]^{\iota^*} & \Hbul(\wtg,k) \ar@{->}[d]^{\iota^*} \ar@{->}[r]^-{\Phi_M} & \Ext_{\wtg}^\bullet(M,M) \ar@{->}[d]^{\iota^*} \\
	S(\g_1^*[p])^{(1)} \ar@{->}[r]^-{\varphi_\g} & \Hbul(\g,k) \ar@{->}[r]^-{\Phi_M} & \Ext_\g^\bullet(M,M),
	}}
	\end{equation}
which in turn implies that $\Chi_\g(M) \subseteq \Chi_{\wtg}(M)$. Then by \eqref{eq:rankinclusion} and Lemma \ref{lemma:supportvarietygradedcase}, we get
	\[
	\Chi_\g'(M) \subseteq \Chi_\g(M) \subseteq \Chi_{\wtg}(M) = \Chi_{\wtg}'(M).
	\]
But $\g_1 = \wtg_1$, so $\Chi_\g'(M) = \Chi_{\wtg}'(M)$, and hence $\Chi_\g'(M) = \Chi_\g(M) = \Chi_{\wtg}(M)$.
\end{proof}
	
\subsection{Support varieties in the general case} \label{subsec:supportvarungraded}

Our goal in this section is to show that the inclusion $\Chi_\g'(M) \subseteq \Chi_\g(M)$ of Proposition \ref{prop:rankinclusion} is an equality.

\begin{theorem}
For each positive integer $m$, one has
	\[
	\Chi_{\glmm}'(\kmm) = \Chi_{\glmm}(\kmm).
	\]
\end{theorem}

\begin{proof}
Let $m$ be a positive integer, let $\g = \glmm$, and let $M = \kmm$ be the natural representation of $\g$. Recall from \cite[\S\S4.2--4.3]{Drupieski:2019a} that the group $G_0 = GL_m(k) \times GL_m(k)$ acts on $\g$ by conjugation, leaving each of $\Chi_\g(k)$, $\Chi_\g(M)$, and $\Chi_\g'(M)$ invariant. Given integers $0 \leq r,s \leq m$, define $x_{r,s} \in \gone$ as in Figure \ref{fig:xmatrix}. Then $\set{x_{r,s}: r+s \leq m}$ is a set of $G_0$-orbit representatives in $\Chi_\g(k)^{(-1)}$, and the orbit closure relations are given by $G_0 \cdot x_{r',s'} \subseteq \ol{G_0 \cdot x_{r,s}}$ if and only if $r' \leq r$ and $s' \leq s$. In particular, the maximal orbits are of the form $G_0 \cdot x_{r,m-r}$ for $0 \leq r \leq m$. By inspection, $\Chi_\g'(M)$ consists of the non-maximal orbits in $\Chi_\g(k)$. Then to show that $\Chi_\g'(M) = \Chi_\g(M)$, it suffices to show for each $0 \leq r \leq m$ that $x_{r,m-r} \notin \Chi_\g(M)^{(-1)}$. Equivalently in the notation of Section \ref{subsec:LSAcohomology}, it suffices to show that there exists a polynomial $f \in S(\gone^*[p])$ such that $\varphi(f^{(1)}) \in I_\g(M)$ but $f(x_{r,m-r}) \neq 0$.

\begin{figure}[htbp]
\scalebox{0.8}{
$\left[
\begin{array}{cccccc|cccccc}
	&	&	&	&	&	& 1	&	&	&	&	&	\\
	&	&	&	&	&	&	& \ddots	&	&	&	& \\
	&	&	&	&	&	&	&	& 1	&	&	& 	\\
	&	&	&	&	&	&	&	&	& 0	&	& 	\\
	&	&	&	&	&	&	&	&	&	& \ddots	& 	\\	
	&	&	&	&	&	&	&	&	&	&	& 0	\\
\hline
0	&	&	&	&	&	&	&	&	&	&	&	\\
	& \ddots	&	&	&	&	&	&	&	&	&	&	\\
	&	& 0	&	&	&	&	&	&	&	&	&	\\
	&	&	& 1	&	&	&	&	&	&	&	&	\\
	&	&	&	& \ddots	&	&	&	&	&	&	&	\\
	&	&	&	&	& 1	&	&	&	&	&	&
\end{array}
\right]
$
}
\caption{The matrix $x_{r,s}$, whose first $r$ diagonal entries in the upper-right $m \times m$ block are equal to $1$, whose last $s$ diagonal entries in the lower-left $m \times m$ block are equal to $1$, and whose remaining entries are $0$.} \label{fig:xmatrix}
\end{figure}

Let $\set{e_1,\ldots,e_{2m}}$ be the standard homogeneous basis for $M$, with $\ol{e_i} = \zero$ if $1 \leq i \leq m$, and $\ol{e_i} = \one$ if $m \leq i \leq 2m$. Fix an integer $0 \leq r \leq m$, and let
	\[
	S = \set{ e_{m+1},\ldots,e_{m+r},e_{r+1},\ldots,e_m}.
	\]
The matrix $x_{r,m-r}$ maps $S$ to the set
	\[
	x_{r,m-r}.S = \set{e_1,\ldots,e_r,e_{m+r+1},\ldots,e_{m+m}},
	\]
so $S$ is a generating set for $M$ as a $\g$-module. Let $F^\bullet M$ be the $\g$-module filtration on $M$ determined by $S$ as in \eqref{eq:modulefiltration}. Then $F^0M = \Span_k(S)$ and $F^1M = M$, so it follows that the associated graded module $\wtM$ is nonzero only in $\Z$-degrees $0$ and $1$, and one gets canonical identifications $\wtM_0 = \Span_k(S)$ and $\wtM_1 = \Span_k(x_{r,m-r}.S)$. Then identifying $x_{r,m-r}$ with an element of $\wtg_1$, the restricted module $\wtM|_{\subgrp{x_{r,m-r}}}$ is free. Since $\wtM|_{\subgrp{x_{r,m-r}}}$ is free, we get by Lemma \ref{lemma:rankequalssupportgradedcase} that $x_{r,m-r} \notin \Chi_{\wtg}(\wtM)^{(-1)}$. Then there exists $f \in S(\wtg_1^*[p])$ such that $\varphi_{\wtg}(f^{(1)}) \in I_{\wtg}(M)$ but $f(x_{r,m-r}) \neq 0$. Identifying $f$ with an element of $S(\gone^*[p])$, this implies by Lemma \ref{lemma:powerannihilates} that $\varphi_\g((f^\ell)^{(1)}) \in I_\g(M)$ for some $\ell \in \N$. Since $f^\ell(x_{r,m-r}) = (f(x_{r,m-r}))^\ell \neq 0$, then $x_{r,m-r} \notin \Chi_\g(M)^{(-1)}$, as desired.
\end{proof}

\begin{corollary} \label{cor:rank=support}
Let $\g$ be a finite-dimensional Lie superalgebra over $k$, and let $M$ be a finite-dimensional $\g$-supermodule. Then the inclusion \eqref{eq:rankinclusion} is an equality:
	\[
	\Chi_\g'(M) = \Chi_\g(M).
	\]
\end{corollary}

\subsection{Applications} \label{subsec:applications}

Our first application is the tensor product property alluded to in the paper's introduction.

\begin{proposition}[Tensor Product Property] \label{cor:tensorproductproperty}
Let $\g$ be a finite-dimensional Lie superalgebra over $k$, and let $M$ and $N$ be finite-dimensional $\g$-supermodules. Then
	\[
	\Chi_\g(M \otimes N) = \Chi_\g(M) \cap \Chi_\g(N).
	\]
\end{proposition}

\begin{proof}
Using Corollary \ref{cor:rank=support}, this can be deduced in the same manner as \cite[Corollary 3.2.4]{Drupieski:2019a}. 
\end{proof}

The next theorem is a positive characteristic analogue of \cite[Theorem 3.4]{Duflo:2005}.

\begin{theorem} \label{theorem:0varietyimpliesfinprojdim}
Let $\g$ be a finite-dimensional Lie superalgebra over an algebraically closed field $k$ of characteristic $p \geq 3$. Let $M$ be a finite-dimensional $\g$-supermodule, and suppose that
	\[
	\set{ x \in \gone : [x,x] = 0 \text{ and } M|_{\subgrp{x}} \text{ is not free}} = \emptyset.
	\]
Then $\projdim_{U(\g)}(M) < \infty$.
\end{theorem}

\begin{proof}
By the hypothesis and Corollary \ref{cor:rank=support}, $\Chi_\g(M) = \set{0}$. Since $\Ext_{\g}^\bullet(M,M)$ is finite over the image of the map $\Phi_M \circ \varphi_{\g}: S(\gone^*[p])^{(1)} \rightarrow \Hbul(\g,k) \rightarrow \Ext_\g^\bullet(M,M)$, this implies that $\Ext_{\g}^\bullet(M,M)$ is finite-dimensional. In particular, there exists $n \in \N$ such that $\Ext_\g^i(M,M) = 0$ for $i > n$.

Now let $N$ be a finitely-generated $U(\g)$-module. By Proposition \ref{prop:EMNfinitenessresults}\eqref{item:Extfinite}, $\Ext_\g^\bullet(M,N)$ is finite under the cup product action of $\Hbul(\g,k)$. The left and right cup product actions differ only by signs, and the right cup product action factors through the map $\Phi_M: \Hbul(\g,k) \rightarrow \Ext_\g^\bullet(M,M)$; see \cite[Proposition 2.3.5]{Drupieski:2019a}. Since $\Ext_\g^i(M,M) = 0$ for $i > n$, this implies that $\Ext_\g^i(M,N) = 0$ for $i \gg 0$. Then Theorem \ref{thm:equivprojdimfinite} and the discussion of Section \ref{subsec:homdimenveloping} implies that $\projdim_{U(\g)}(M) < \infty$.
\end{proof}

The following result was proved for finite-dimensional positively-graded Lie super\-algebras by B{\o}gvad \cite{Bo-gvad:1984}, and was proved for arbitrary finite-dimensional Lie superalgebras in char\-acteristic zero by Musson \cite[Theorem 17.1.2]{Musson:2012}.

\begin{corollary} \label{C:finiteglobaldimension}
Let $\g$ be a finite-dimensional Lie superalgebra over a field $k$ of characteristic $p \geq 3$. Let $K = \ol{k}$ be the algebraic closure of $k$, and suppose that $\g_K \colonequals \g \otimes_k K$ is torsion free, i.e., if $x \in (\g_K)_{\one} = (\gone)_K$ and $[x,x]=0$, then $x=0$. Then $U(\g)$ has finite global dimension.
\end{corollary}

\begin{proof}
Evidently, $U(\g_K) \cong U(\g) \otimes_k K$. Let $\set{e_i : i \in I}$ be an $k$-basis for $K$ containing the identity element $1_k = 1_K$. Then $U(\g_K) = \bigoplus_{i \in I} U(\g) \otimes_k e_i$ is a $U(\g)$-bimodule decomposition of $U(\g_K)$. In particular, $U(\g)$ is bimodule direct summand of $U(\g_K)$, and $U(\g_K)$ is free over $U(\g)$. Then $\gldim(U(\g)) \leq \gldim(U(\g_K))$ by \cite[Theorem 7.2.8]{McConnell:2001}, so it suffices to assume that $k = K$. Now the hypothesis implies that $\Chi_\g(k) = \set{0}$, meaning that $\projdim_{U(\g)}(k) < \infty$ by Theorem \ref{theorem:0varietyimpliesfinprojdim}. Then $\gldim(U(\g)) < \infty$ by \cite[Corollary 1.4(c)]{Brown:1997}.
\end{proof}

\appendix

\section{Homological dimensions over Noether algebras} \label{SS:homologicaldimension}

\begin{center}
(by Luchezar L.\ Avramov\footnote{Luchezar L.\ Avramov: Department of Mathematics, University of Nebraska, Lincoln, NE 68588, USA. E-mail address: avramov@math.unl.edu} and Srikanth B.\ Iyengar\footnote{Srikanth B.\ Iyengar: Department of Mathematics, University of Utah, 155 South 1400 East, Salt Lake City, UT, 84112-0090, USA. Email address: iyengar@math.utah.edu})
\end{center}

\subsection{}

When $M$ is a (say, left) module over an associative ring $A$, the condition
	\begin{center}
	 $\Ext^i_A(M,N)=0$ for some \emph{fixed} $i \geq 1$ and for \emph{all} $A$-modules $N$
	\end{center}
is equivalent to an \emph{upper bound} $\projdim_A(M)<i$ on the projective dimension. We identify hypoth\-eses under which two weaker conditions are equivalent; namely
	\begin{center}
	$\Ext^i_A(M,N)=0$ for \emph{all finite} $A$-modules $N$ and for \emph{all sufficiently large} $i$
	\end{center}
is equivalent to the \emph{finiteness} of $\projdim_A(M)$.

\begin{definition}[Noether algebra] \label{def:noetheralgebra}
A \emph{Noether $R$-algebra} $A$ is an associative ring that is finite (that is to say, finitely-generated) as a module over a noetherian ring $R$ lying in the center of $A$. If no central subring is identified, we simply say that $A$ is a Noether algebra.
\end{definition}

Noether algebras are left- and right-noetherian rings. The case of an artinian central subring yields the widely used concept of Artin algebra.

The following theorem is a special case, with a different proof, of a result in \cite{Avramov:2014}.  Only the last condition explicitly invokes the ring $R$: If $\fp$ is a prime ideal of $R$, then localization with respect to the (central) multiplicatively closed set $R\setminus \fp$ turns $A_{\fp}$ into an $R_{\fp}$-algebra and $M_{\fp}$ into an $A_{\fp}$-module.

\begin{theorem} \label{thm:equivprojdimfinite}
Let $A$ be a Noether $R$-algebra and $M$ a finite $A$-module. The following conditions are equivalent:
	\begin{enumerate}
	\item \label{item:projdimfin} $\projdim_A(M) < \infty$.
	\item \label{item:Extvanish} $\Ext_A^i(M,N) = 0$ for all finite $A$-modules $N$ and $i \gg 0$.
	\item \label{item:Extvanishlocal} $\Ext_{A_{\fm}}^i(M_{\fm},L) = 0$ for all maximal ideals $\fm$ 
	of $R$, simple $A_{\fm}$-modules $L$, and $i \gg 0$.
	\end{enumerate}
\end{theorem} 

The implication \eqref{item:projdimfin}$\Rightarrow$\eqref{item:Extvanish} is clear, and \eqref{item:Extvanish}$\Rightarrow$\eqref{item:Extvanishlocal} follows from basic properties of localization. To close the loop, we draw on a classical \emph{local} characterization of finite projective dimension: 

\begin{remark} \label{rem:equivprojdimfinite}
For $R$, $A$, and $M$ as in the theorem, the following conditions are equivalent:
	\begin{enumerate}
	\item \label{remitem:projdimfin} $\projdim_A(M) < \infty$.
	\item \label{item:localmaxprojdimfin} $\projdim_{A_\fm}(M_{\fm}) < \infty$ for all maximal ideals $\fm$ of $R$.
	\end{enumerate}

Once again, the implication \eqref{remitem:projdimfin}$\Rightarrow$\eqref{item:localmaxprojdimfin} reflects the exactness of the localization functor. The converse is due to Bass and Murthy; see \cite[Lemma 4.5]{Bass:1967} or \cite[Corollary III.6.6]{Bass:1968}.
\end{remark} 

In the proofs we use the canonical action of $R$ on $\Ext$: Multiplication by $r \in R$ on $\Ext_{A}^i(M,N)$ is the map induced by the $A$-linear (because $R$ is central) endomorphism of $N$, given by $n \mapsto rn$.

For convenience, we include a variation of a standard result on flat base change.

\begin{lemma} \label{lemma:ExtAisRnoetherian}
Let $A$ be a Noether $R$-algebra, $R \rightarrow R'$ a homomorphism of commutative rings, and $(-)'$ the functor $R' \otimes_R -$.  If $R'$ is flat as an $R$-module and if $M$ is a finite $A$-module, then for each $A$-module $N$ there are natural isomorphisms
	\[
	R'\otimes_R \Ext_A^i(M,N) \cong \Ext_{A'}^i(M', N')  \quad\text{for } i \in \Z.
	\]
\end{lemma}

\begin{proof}
Let $P_\bullet$ be a projective resolution of $M$ with $P_i = A^{\oplus b_i}$, $b_i \in \N$, for each $i$. The map $h^\bullet: \Hom_A(P_\bullet,N) \rightarrow \Hom_{A'}(P'_\bullet,N')$, which takes a homomorphism of complexes $P_\bullet \rightarrow N$ to the induced map $P_\bullet' \rightarrow N'$, is a morphism of complexes over $R$. Its target is a complex of $R'$-modules, so $h^\bullet$ factors uniquely through a morphism $\wt{h}^\bullet : \Hom_A(P_\bullet,N)' \rightarrow \Hom_{A'}(P'_\bullet,N')$ of such complexes, and $\wt{h}^i$ is the canonical isomorphism $\Hom_A(A^{\oplus b_i},N)' \cong \Hom_{A'}((A^{ \oplus b_i})',N')$. Since $P'_\bullet$ is a projective resolution of $M'$ over $A'$, the maps $\opH^i(\wt{h})$ yield the desired isomorphisms.
\end{proof}

Recall that an associative ring is said to be \emph{semilocal} if the residue ring modulo its Jacobson radical is semisimple; in the commutative case, this defines the rings with finitely many maximal ideals; see \cite[Proposition 20.2]{Lam:2001}. Semilocal rings have finitely many simple modules.  

We say that a Noether $R$-algebra $A$ is \emph{semilocal} if the ring $R$ has this property; it is well-known (for reasons recalled in the next proof) that such algebras are semilocal rings.

\begin{proposition} \label{prop:semilocal}
Let $A$ be a semilocal Noether $R$-algebra, and set
	\[
	p_{A}(M) \colonequals \max_{1\le j\le r} \{ i \in \N \mid \Ext_{A}^i(M,L_j) \neq 0 \},
	\]
where $L_1,\dots,L_r$ are the simple $A$-modules. Then the following equality holds:
	\[
	\projdim_A(M) = p_{A}(M).
	\]
\end{proposition}

\begin{proof}
Let $\fm$ denote the Jacobson ideal of $R$, and set $k \colonequals R/\fm$ and $\ol{A} \colonequals k\otimes_R A$. The Jacobson radical $J$ of the ring $A$ satisfies $J \supseteq \fm A \supseteq J^n$ for some $n \geq 1$, because $A$ is finite as an $R$-module; see \cite[Proposition 20.6]{Lam:2001}. It follows that $A$ is a semilocal ring, with the same simple modules as $\ol{A}$. 

Let $\wh{R}$ be the $\fm$-adic completion of $R$; recall that $\wh{R}$ is a noetherian semilocal ring with Jacobson radical $\fm \wh{R}$, and the completion map $R \rightarrow \wh{R}$ is a faithfully flat ring homomorphism that induces an isomorphism of $R/\fm$ and $\wh{R}/\fm \wh{R}$. Set $\wh{A} \colonequals \wh{R} \otimes_R A$, and $\wh{N} \colonequals \wh{R} \otimes_R N$ for each $R$-module $N$. The induced ring homomorphism $\wh{R} \rightarrow \wh{A}$ is injective; it turns $\wh{A}$ into a Noether $\wh{R}$-algebra and $\wh{M}$ into a finite $\wh{A}$-module. Furthermore, the induced homomorphism of $k$-algebras $\ol{A} =k \otimes_{R} A \rightarrow k \otimes_{\wh{R}}\wh{A}$ is bijective and so are the induced equivariant maps $L_i= k \otimes_{R} L_i \rightarrow k\otimes_{\wh{R}} \wh{L}_i$. It follows that the ring $\wh{A}$ is semilocal, with simple modules $\wh{L}_1,\ldots,\wh{L}_r$. Lemma \ref{lemma:ExtAisRnoetherian} provides isomorphism
	\[
	\wh{R} \otimes_R \Ext_{A}^i(M,N) \cong \Ext_{\wh{R} \otimes_R A}^i(\wh{R} \otimes_R M, \wh{R} \otimes_R N)  \quad\text{for } i \in \Z.
	\]
As the functor $\wh{R} \otimes_R -$ is faithful, and $\projdim_A(M)$ equals $\sup_N \{ i \in \N \mid \Ext_{A}^i(M,N) \neq 0 \}$ when $N$ ranges over finitely generated $A$-modules (see \cite[Proposition VI.2.5]{Cartan:1999}), the isomorphisms yield
	\[
	\projdim_{\wh{A}}(\wh{M}) \geq \projdim_A(M) \geq p_{A}(M) = p_{\wh{A}}(\wh{M}).
	\]
The upshot of the preceding discussion is that it suffices to prove that $p_{A}(M)$ equals $\projdim_A(M)$ under the additional hypothesis that $R$ is $\fm$-adically complete. The benefit is that then the ring $A$ is semiperfect; see \cite[Example 23.3]{Lam:2001}. As it is also noetherian, the finite module $M$ has a projective resolution $P_\bullet$ in which every module $P_i$ is finite projective and every differential $\partial_i \colon P_i \rightarrow P_{i-1}$ satisfies $\partial_i(P_i) \subseteq JP_{i-1}$; see \cite[Proposition 24.12]{Lam:2001}. Such a \emph{minimal resolution} yields isomorphisms
	\[
	\Hom_A(P_i, L_j) \cong\Ext_{A}^i(M,L_j) \quad \text{for } 1 \leq j \leq r \text{ and } i \in \Z.
	\]
As $\Hom_A(P_i, L_j)=0$ for $1 \leq j \leq r$ implies $P_i = 0$, we get $\projdim_A(M) = p_{A}(M)$, as desired.
\end{proof}

\begin{proof}[Proof of Theorem \emph{\ref{thm:equivprojdimfinite}}]
\eqref{item:Extvanish}$\Rightarrow$\eqref{item:Extvanishlocal}.
Let $\fm$ be a maximal ideal of $R$ and $L$ a simple $A_{\fm}$-module. For any nonzero $l \in L$ we have $L=A_{\fm}l\cong (Al)_{\fm}$. Since the functors $(-)_{\fm}$ and $R_{\fm}\otimes_R-$ are isomorphic, and $R_\fm$ is flat as an $R$-module, from Lemma \ref{lemma:ExtAisRnoetherian} and the hypothesis we get
	\[
	\Ext_{A_{\fm}}^i(M_{\fm}, L) \cong \Ext_{A_{\fm}}^i(M_{\fm}, (Al)_\fm) \cong \Ext_A^i(M,Al)_{\fm} = 0 \quad \text{for } i \gg 0.
	\]

\eqref{item:Extvanishlocal}$\Rightarrow$\eqref{item:projdimfin}.
Due to the hypothesis, Proposition \ref{prop:semilocal} shows that $\projdim_{A_\fm}(M_{\fm})$ is finite for every maximal ideal $\fm$ of $R$, and by Remark \ref{rem:equivprojdimfinite} this means that $\projdim_A(M)$ is finite.
\end{proof}

\makeatletter
\renewcommand*{\@biblabel}[1]{\hfill#1.}
\makeatother

\bibliographystyle{eprintamsplain}
\bibliography{support-varieties-modular-lie-superalegbras}

\end{document}